\documentclass[12pt]{amsart}
\usepackage{amssymb}
\usepackage{amsfonts}
\usepackage{amssymb,latexsym}
\usepackage{enumerate}
\usepackage{mathrsfs}
\usepackage{url}

\makeatletter
\@namedef{subjclassname@2020}{%
  \textup{2020} Mathematics Subject Classification}
\makeatother

  [2002/01/19 v2.2g %
    AMS font definitions%
  ]
\DeclareFontFamily{U}{euf}{}
\DeclareFontShape{U}{euf}{m}{n}{%
  <5><6><7><8><9>gen*eufm%
  <10><10.95><12><14.4><17.28><20.74><24.88>eufm10%
  }{}
\DeclareFontShape{U}{euf}{b}{n}{%
  <5><6><7><8><9>gen*eufb%
  <10><10.95><12><14.4><17.28><20.74><24.88>eufb10%
  }{}

  [2002/01/19 v2.2g %
    AMS font definitions%
  ]
\DeclareFontFamily{U}{msb}{}
\DeclareFontShape{U}{msb}{m}{n}{%
  <5><6><7><8><9>gen*msbm%
  <10><10.95><12><14.4><17.28><20.74><24.88>msbm10%
  }{}

  [2002/01/19 v2.2g %
    AMS font definitions%
  ]
\DeclareFontFamily{U}{msa}{}
\DeclareFontShape{U}{msa}{m}{n}{%
  <5><6><7><8><9>gen*msam%
  <10><10.95><12><14.4><17.28><20.74><24.88>msam10%
  }{}

\newtheorem{theorem}{Theorem}

\newtheorem{corollary}[theorem]{Corollary}

\theoremstyle{definition}
\newtheorem{remark}[theorem]{Remark}

\newcommand{\abs}[1]{\lvert#1\rvert}

\begin{document}

\title[generalizations of Wallis' formula]
{Euler's transformation, zeta functions and generalizations of Wallis' formula}

 \author{Qianqian Cai}
\address{Department of Mathematics, South China University of Technology, Guangzhou, Guangdong 510640, China}
\email{jenny.royce@foxmail.com}

\author{Su Hu}
\address{Department of Mathematics, South China University of Technology, Guangzhou, Guangdong 510640, China}
\email{mahusu@scut.edu.cn}

\author{Min-Soo Kim}
\address{Department of Mathematics Education, Kyungnam University, Changwon, Gyeongnam 51767, Republic of Korea}
\email{mskim@kyungnam.ac.kr}

\subjclass[2020]{11M06; 11B68; 11Y60; 40G05; 26B20}
\keywords{Euler's transformation, Zeta function, Wallis' formula, Euler-Mascheroni constant, Glaisher-Kinkelin constant}

\maketitle

\begin{abstract}
In this note, we extend Euler's transformation formula from the alternating series
to more general series.
Then we give new expressions for the Riemann zeta  function $\zeta(s)$ by the generalized difference operator $\Delta_{c}$,
which provide analytic continuation of $\zeta(s)$ and new ways to evaluate the special values of $\zeta(-m)$ for $m=0,1,2,\ldots$.
Applying these results, we further extend Huylebrouck's generalization of Wallis'  well-known formula for $\pi$ in the half 
planes Re$(s)>0$ and Re$(s)>-1$, respectively. They imply several interesting special cases including
$$
\frac{2\pi}{3^{\frac{3}{2}}}=\frac{3^{\frac{4}{3}}}{2^{\frac{4}{3}}}
\frac{2^{\frac{1}{3}}\cdot3^{\frac{1}{3}}\cdot3^{\frac{1}{3}}\cdot4^{\frac{1}{3}}\cdot6^{\frac{2}{3}}\cdot6^{\frac{2}{3}}}{4^{\frac{1}{3}}\cdot4^{\frac{1}{3}}\cdot5^{\frac{1}{3}}\cdot5^{\frac{1}{3}}\cdot4^{\frac{2}{3}}\cdot5^{\frac{2}{3}}}\cdots,
$$
$$
 3^{\gamma-\frac{\log 3}{2}}=\frac{3^{\frac{1}{3}}\cdot3^{\frac{1}{3}}}{2^{\frac{1}{2}}\cdot4^{\frac{1}{4}}} \frac{6^{\frac{1}{6}}\cdot6^{\frac{1}{6}}}{5^{\frac{1}{5}}\cdot7^{\frac{1}{7}}}\frac{9^{\frac{1}{9}}\cdot9^{\frac{1}{9}}}{8^{\frac{1}{8}}\cdot10^{\frac{1}{10}}}\cdots,
$$
and
$$
\left(3\left(\frac{2\pi e^{\gamma}}{A^{12}}\right)^{2}\right)^{\frac{\pi^2}{18}}=\frac{3^{\frac{1}{3^2}}\cdot3^{\frac{1}{3^2}}}{2^{\frac{1}{2^2}}\cdot4^{\frac{1}{4^2}}} \frac{6^{\frac{1}{6^2}}\cdot6^{\frac{1}{6^2}}}{5^{\frac{1}{5^2}}\cdot7^{\frac{1}{7^2}}}\frac{9^{\frac{1}{9^2}}\cdot9^{\frac{1}{9^2}}}{8^{\frac{1}{8^2}}\cdot10^{\frac{1}{10^2}}}\cdots,$$
where $\gamma$ is the Euler-Mascheroni constant and $A$ is the Glaisher-Kinkelin constant.
\end{abstract}

\section{Introduction}
The Riemann zeta function is defined by
\begin{equation}~\label{Ri-zeta}
\zeta(s)=\sum_{n=1}^{\infty}\frac{1}{n^{s}},
\end{equation}
for Re$(s)>1$.
It can be analytically continued to a meromorphic function in the
complex plane with a simple pole at $s=1$.
And the Dirichlet eta function is an alternating form of $\zeta(s)$, 
 \begin{equation}~\label{Euler}
\eta(s)=\sum_{n=1}^{\infty}\frac{(-1)^{n-1}}{n^{s}},
\end{equation}
for Re$(s)>0$.
It  can be analytically continued  to the complex plane without any pole. 
For Re$(s)>0$, (\ref{Ri-zeta}) and (\ref{Euler}) are connected by the following equation
\begin{equation}~\label{Riemann-Euler}
\eta(s)=(1-2^{1-s})\zeta(s).
\end{equation}
The Dirichlet eta function $\eta(s)$ appeared as a tool for Euler's derivation of the functional equation for $\zeta(s)$.
In fact, according to Weil's history~\cite[p.~273--276]{Weil},
Euler  ``proved"
\begin{equation}~\label{fe}
	\frac{\eta(1-s)}{\eta(s)}=\frac{-\Gamma(s)(2^{s}-1)\textrm{cos}(\pi s/2)}{(2^{s-1}-1)\pi^{s}},
\end{equation}
then from  (\ref{Riemann-Euler}) he got the functional equation of $\zeta(s)$.

In order to apply $\eta(s)$ to calculate the special values $\zeta(-m)$ 
for $m=0,1,2,\ldots,$  Euler introduced the following transformation of alternating series (see \cite[volume 10, p. 222--227]{Euler}).
Let $A_2=\sum_{n=1}^{\infty}(-1)^{n+1}b_n$ be an alternating series (see \eqref{gseries}). It can be written as
\begin{equation}\label{c2}
\begin{aligned}
A_2=&b_1-b_2+b_3-b_4+\cdots\\
=&\frac{1}{2} b_{1}+\frac{1}{2}\left[\left(b_{1}-b_{2}\right)-\left(b_{2}-b_{3}\right)+\cdots\right] \\
=&\frac{1}{2} b_{1}+\frac{1}{4}\left(b_{1}-b_{2}\right)+\frac{1}{4}\left[\left(b_{1}-2 b_{2}+b_{3}\right)-\left(b_{2}-2 b_{3}+b_{4}\right)+\cdots\right].
\end{aligned}
\end{equation}
So inductively, in general, we have
\begin{equation}\label{c2k}
	\sum_{n=1}^{\infty}(-1)^{n+1} b_{n}=\sum_{j=0}^{k-1} \frac{\Delta^{j} b_{1}}{2^{j+1}}+\sum_{n=1}^{\infty}(-1)^{n+1} \frac{\Delta^{k} b_{n}}{2^{k}},
\end{equation}
where the sequence of difference operators $\{\Delta^{k}\}_{k=1}^{\infty}$ is
defined recursively by $\Delta^{0} b_{n}=b_{n}$ and
\begin{equation}\label{d2}
	\Delta^{k} b_{n}=\Delta^{k-1} b_{n}-\Delta^{k-1} b_{n+1}
\end{equation}
for $k \geq 1$.
In 1994, by using Euler's transformation, Sondow \cite{Sondow2} obtained a new expression for $\eta(s),$ which implies an analytic continuation of
the Riemann zeta function $\zeta(s)$ to complex numbers $s\neq 1$. His main result is as follows.
\begin{theorem}[{Sondow \cite[p. 423, (8)]{Sondow2}}]\label{Sondow-the1}
	For $k \geq 1$, the expression
	\begin{equation}\label{Sondow1}
	\begin{aligned}
		\eta(s)&=\left(1-2^{1-s}\right)\zeta(s)\\
		&=\sum_{j=0}^{k-1} \frac{\Delta^{j} 1^{-s}}{2^{j+1}}+\frac{1}{2^{k}} \sum_{n=1}^{\infty}(-1)^{n-1} \Delta^{k} n^{-s},
		\end{aligned}
	\end{equation}
	provides the analytic continuation of $\zeta(s)$ on the punctured half plane Re$(s)>1-k$, $s \neq 1$ where the infinite series converges absolutely and uniformly on compact sets to a holomorphic function. Moreover, except that the convergence will not be absolute in the strip $-k<\text{Re}(s) \leq 1-k$, this remains true for $k \geq 0$ and $\text{Re}(s)>-k$, $s\neq1$. Especially, taking $k=1$, we have
	\begin{equation}\label{Sondow2}
	\begin{aligned}
		\eta(s)&=\left(1-2^{1-s}\right)\zeta(s)\\
		&=\frac{1}{2}+\frac{1}{2} \sum_{n=1}^{\infty}(-1)^{n-1}\left(n^{-s}-(n+1)^{-s}\right)
		\end{aligned}
	\end{equation}
	for Re$(s)>-1$ and $s \neq 1$.
\end{theorem}
From this he successfully  got a new expression of $\zeta(-m)$ for $m=0,1,2,\ldots,$ 
\begin{equation}\label{s1}
	\zeta(-m)=\frac{1}{1-2^{m+1}} \sum_{j=0}^{m} \frac{\Delta^{j} 1^{m}}{2^{j+1}}
\end{equation}
(see \cite[p. 423, Corollary]{Sondow2}).

Euler's transformation and Sondow's result show a connection between Wallis' well-known formula for $\pi$ and zeta functions.
In a book published in 1656 \cite{Wallis}, John Wallis presented the following remarkable infinite product
representation of $\pi$
\begin{equation}\label{Wallis} 
\frac{\pi}{2}=\frac{2\cdot2}{1\cdot 3}\frac{4\cdot4}{3\cdot 5}\frac{6\cdot6}{5\cdot 7}\cdots,
\end{equation}
which has been quoted by many calculus books. Well-known proofs include an application of the formula for integrals of powers of $\sin x$ from the inductive method 
or an application of the infinite product expansion of $\sin x$. 
Taking derivatives on both sides of (\ref{Sondow2}) we obtain
\begin{equation}
\eta^{\prime}(0)=\frac{1}{2}\log\left(\frac{2\cdot 2}{1\cdot 3}\frac{4\cdot4}{3\cdot 5}\frac{6\cdot 6}{5\cdot 7}\cdots\right).
\end{equation}
From  this, Yung (in 1999) and Sondow (in 2002) found a new proof of Wallis'  formula (see \cite{Sondow}).

  In 2015, by extending Yung and Sondow's methods to other values of $s$, Huylebrouck~\cite{Huy} obtained the following result.
\begin{theorem}[{Huylebrouck \cite[p. 371, Theorem 1]{Huy}}]
For appropriate values of $s$ (and if $\sqrt[n^{0}]{n}$ is interpreted as $n$),
\begin{equation}\label{Theorem1}
e^{2\eta^{\prime}(s)}=\frac{\sqrt[2^s]{2}\cdot\sqrt[2^s]{2}}{\sqrt[1^s]{1}\cdot \sqrt[3^s]{3}}\frac{\sqrt[4^s]{4}\cdot\sqrt[4^s]{4}}{\sqrt[3^s]{3}\cdot \sqrt[5^s]{5}}\frac{\sqrt[6^s]{6}\cdot\sqrt[6^s]{6}}{\sqrt[5^s]{5}\cdot \sqrt[7^s]{7}}\cdots.\end{equation}
\end{theorem}
The $s=0$ case of the above result  recovers Wallis'  formula,
the $s=1$ case implies a first generalization of Wallis' formula:
\begin{equation}\label{first}
2^{(2\gamma-\log 2)}=\frac{2^{\frac{1}{2}}\cdot 2^{\frac{1}{2}}}{1^{\frac{1}{1}}\cdot 3^{\frac{1}{3}}}\frac{4^{\frac{1}{4}}\cdot 4^{\frac{1}{4}}}{3^{\frac{1}{3}}\cdot 5^{\frac{1}{5}}}\frac{6^{\frac{1}{6}}\cdot 6^{\frac{1}{6}}}{5^{\frac{1}{5}}\cdot 7^{\frac{1}{7}}}\cdots,\end{equation}
where $\gamma=0.5772156649\cdots$ is the Euler-Mascheroni constant,
and the  $s=2$  case implies the second generalization of Wallis' formula:
\begin{equation}\label{second}
\left(\frac{4\pi e^{\gamma}}{A^{12}}\right)^{\frac{\pi^2}{6}}=\frac{2^{\frac{1}{2^2}}\cdot 2^{\frac{1}{2^2}}}{1^{\frac{1}{1^2}}\cdot 3^{\frac{1}{3^2}}}\frac{4^{\frac{1}{4^2}}\cdot 4^{\frac{1}{4^2}}}{3^{\frac{1}{3^2}}\cdot 5^{\frac{1}{5^2}}}\frac{6^{\frac{1}{6^2}}\cdot 6^{\frac{1}{6^2}}}{5^{\frac{1}{5^2}}\cdot 7^{\frac{1}{7^2}}}\cdots,\end{equation} 
where $A=1.2824271291\cdots$ is the Glaisher-Kinkelin constant (see \cite{Weisstein} for the definition).

In this note, we go to a more general case.
Let $c\geq 2$ be an integer, for Re$(s)>0,$ if denote by
\begin{equation}\label{A}
	\zeta_{(c)}(s)=(1-c^{1-s})\zeta(s),
\end{equation}
then  we have 
\begin{equation}\label{B}
	\zeta_{(c)}(s)=\sum_{n=1}^{\infty}\frac{a_{c,n}}{n^s},
\end{equation}
where 
\begin{equation}\label{C}
	a_{c,n}=\begin{cases}
		1-c&\textrm{if}~n\equiv 0~(\textrm{mod}~c);\\
		1&\textrm{if}~n\not\equiv 0~(\textrm{mod}~c)\end{cases}
\end{equation}
(see \cite[p. 326]{KW}).

Inspiring by the representation of $\zeta_{(c)}(s)$ (see (\ref{B})), we extend Euler's transformation formula from the alternating series
to more general series with the form $A_{c}=\sum_{n=1}^{\infty}a_{c,n}b_n$, where $a_{c,n}$ is given in (\ref{C}). (See Theorem \ref{Euler1}).
Then we generalize the above Sondow's result to give expressions for $\zeta_{(c)}(s)$ by the generalized difference operator $\Delta_{c}$.    
It provides analytic continuation of the Riemann zeta function $\zeta(s)$ and new ways to evaluate $\zeta(-m)$ for $m=0,1,2,\ldots$ (see Theorem \ref{gsondow} and Corollary \ref{corollaryc}).
Based on these results, with the help of
$\zeta_{(c)}(s)$ we further extend Huylebrouck's generalization of Wallis' formula in the half planes Re$(s)>0$ and Re$(s)>-1$, respectively (see Theorems \ref{main} and \ref{re-1}), which imply several interesting special cases including
$$
\frac{2\pi}{3^{\frac{3}{2}}}=\frac{3^{\frac{4}{3}}}{2^{\frac{4}{3}}}
\frac{2^{\frac{1}{3}}\cdot3^{\frac{1}{3}}\cdot3^{\frac{1}{3}}\cdot4^{\frac{1}{3}}\cdot6^{\frac{2}{3}}\cdot6^{\frac{2}{3}}}{4^{\frac{1}{3}}\cdot4^{\frac{1}{3}}\cdot5^{\frac{1}{3}}\cdot5^{\frac{1}{3}}\cdot4^{\frac{2}{3}}\cdot5^{\frac{2}{3}}}\cdots,
$$
$$
 3^{\gamma-\frac{\log 3}{2}}=\frac{3^{\frac{1}{3}}\cdot3^{\frac{1}{3}}}{2^{\frac{1}{2}}\cdot4^{\frac{1}{4}}} \frac{6^{\frac{1}{6}}\cdot6^{\frac{1}{6}}}{5^{\frac{1}{5}}\cdot7^{\frac{1}{7}}}\frac{9^{\frac{1}{9}}\cdot9^{\frac{1}{9}}}{8^{\frac{1}{8}}\cdot10^{\frac{1}{10}}}\cdots,
$$
and
$$
\left(3\left(\frac{2\pi e^{\gamma}}{A^{12}}\right)^{2}\right)^{\frac{\pi^2}{18}}=\frac{3^{\frac{1}{3^2}}\cdot3^{\frac{1}{3^2}}}{2^{\frac{1}{2^2}}\cdot4^{\frac{1}{4^2}}} \frac{6^{\frac{1}{6^2}}\cdot6^{\frac{1}{6^2}}}{5^{\frac{1}{5^2}}\cdot7^{\frac{1}{7^2}}}\frac{9^{\frac{1}{9^2}}\cdot9^{\frac{1}{9^2}}}{8^{\frac{1}{8^2}}\cdot10^{\frac{1}{10^2}}}\cdots,$$
where $\gamma$ is the Euler-Mascheroni constant and $A$ is the Glaisher-Kinkelin constant (see the last section).

\section{Euler's transformation and zeta functions}
In this section, we shall generalize Euler's transformation from the alternating series to  a more general setting. Then applying  this, we give new expressions for $\zeta_{(c)}(s),$ which  
provide analytic continuation of the Riemann zeta function $\zeta(s)$ and new ways to evaluate $\zeta(-m)$ for $m=0,1,2,\ldots$.

Let 
\begin{equation}\label{gseries} 
A_c=\sum_{n=1}^{\infty}a_{c,n}b_n
\end{equation} 
be a complex series, where $a_{c,n}$ is defined in (\ref{C}).

For $c=2,$ the series (\ref{gseries}) is just the alternating series which can also be written in  the following form
\begin{align*}
	A_2=&b_1-b_2+b_3-b_4+b_5-b_6+\cdots\\
	=&\frac{1}{2}(2b_1-b_2)-\frac{1}{2}[(b_2-b_3)-(b_3-b_4)+(b_4-b_5)-(b_5-b_6)+\cdots]\\
	=&\frac{1}{2}(2b_1-b_2)-\frac{1}{4}(2b_2-3b_3+b_4)+\frac{1}{4}[(b_3-2b_4+b_5)-(b_3-2b_4+b_5)+\cdots].
\end{align*}
In general, for $k\geq1$ we have
\begin{equation}
	\sum_{n=1}^{\infty}a_{n,2}b_n=\sum_{j=0}^{k-1}\frac{(-1)^j}{2^{j+1}}\left(2\Delta_2^jb_1-\Delta_2^jb_2\right)+\sum_{n=1}^{\infty}\frac{(-1)^k}{2^k}a_{n,2}\Delta_2^kb_n,
\end{equation}
where $\Delta_2^0b_n=b_n$, $\Delta_2^kb_n=\Delta_2^{k-1}b_{n+1}-\Delta_2^{k-1}b_{n+2}$ for $k\geq1$. %This is  (\ref{anc}) in the special case of $c=2$.
It needs to mention that here $\Delta_{2}^{1}b_{n}=b_{n+1}-b_{n+2}$, which is slightly different from the difference operator in Euler's transformation. In that case,
we have $\Delta^{1}b_{n}=b_{n}-b_{n+1}$ (see (\ref{d2})).

For $c=3$, the series  (\ref{gseries}) becomes to
\begin{align*}
	A_3=&b_1+b_2-2b_3+b_4+b_5-2b_6+\cdots\\
	=&\frac{1}{3}(3b_1+4b_2-4b_3)-\frac{1}{3}[(b_2+b_3-2b_4)+(b_3+b_4-2b_5)-2(b_4+b_5-2b_6)+\cdots]\\
	=&\frac{1}{3}(3b_1+4b_2-4b_3)-\frac{1}{9}(3b_2+7b_3-6b_4-12b_5+8b_6)+\frac{1}{9}[(b_3+2b_4-3b_5\\&-4b_6+4b_7)+(b_4+2b_5-3b_6-4b_7+4b_8)-2(b_5+2b_6-3b_7-4b_8+4b_9)+\cdots].
\end{align*}
In general, for $k\geq1$ we have
\begin{equation}
	\sum_{n=1}^{\infty}a_{n,3}b_n=\sum_{j=0}^{k-1}\frac{(-1)^j}{3^{j+1}}\left(3\Delta_3^jb_1+4\Delta_3^jb_2-4\Delta_3^jb_3\right)+\sum_{n=1}^{\infty}\frac{(-1)^k}{3^k}a_{n,3}\Delta_3^kb_n,
\end{equation}
where $\Delta_3^0b_n=b_n$, $\Delta_3^kb_n=\Delta_3^{k-1}b_{n+1}+\Delta_3^{k-1}b_{n+2}-2\Delta_3^{k-1}b_{n+3}$. %This is  (\ref{anc}) in the special case of $c=3$.

For $c=4$, the series  (\ref{gseries}) becomes to
\begin{align*}
	A_4=&b_1+b_2+b_3-3b_4+b_5+b_6+b_7-3b_8\cdots\\
	=&\frac{1}{4}(4b_1+5b_2+6b_3-9b_4)-\frac{1}{4}[(b_2+b_3+b_4-3b_5)+(b_3+b_4+b_5-3b_6)\\&+(b_4+b_5+b_6-3b_7)-3(b_5+b_6+b_7-3b_8)+\cdots].
\end{align*}
In general, for $k\geq1$ we have
\begin{equation}
	\sum_{n=1}^{\infty}a_{n,4}b_n=\sum_{j=0}^{k-1}\frac{(-1)^j}{4^{j+1}}\left(4\Delta_4^jb_1+5\Delta_4^jb_2+6\Delta_4^jb_3-9\Delta_4^jb_4\right)+\sum_{n=1}^{\infty}\frac{(-1)^k}{4^k}a_{n,4}\Delta_4^kb_n,
\end{equation}
where $\Delta_4^0b_n=b_n$, $\Delta_4^kb_n=\Delta_4^{k-1}b_{n+1}+\Delta_4^{k-1}b_{n+2}+\Delta_4^{k-1}b_{n+3}-3\Delta_4^{k-1}b_{n+4}$. %This is  (\ref{anc}) in the special case of $c=4$.

In general, for arbitrary $c$, we have the following result, in which, a generalized difference operator $\Delta_{c}$ is introduced.
\begin{theorem}\label{Euler1}
	For $c\geq2$ and $k\geq1$, we have
	\begin{equation}\label{anc}
		\begin{aligned}
			A_c
			&=\sum_{j=0}^{k-1}\frac{(-1)^j}{c^{j+1}}\left(\sum_{i=0}^{c-2}(c+i)\Delta_c^jb_{i+1}-(c-1)^2\Delta_c^jb_c\right)\\
			&\quad+\sum_{n=1}^{\infty}\frac{(-1)^k}{c^k}a_{c,n}\Delta_c^kb_n,
		\end{aligned}
	\end{equation}
	where the sequence of generalized difference operators $\{\Delta_{c}^{k}\}_{k=1}^{\infty}$ is defined recursively by $\Delta_c^0b_n=b_n$ and 
	\begin{equation}\label{gc} 
		\Delta_c^kb_n=\sum_{i=1}^{c}a_{c,i}\Delta_c^{k-1}b_{n+i}
	\end{equation} for $k\geq1$. Furthermore, if the series (\ref{gseries}) is convergent, so does the right hand side of (\ref{anc}).
\end{theorem}

\begin{remark}
	In some cases, although (\ref{gseries}) is divergent, the series on the right hand side of  (\ref{anc}) may still converge, so we can endow its sum to  (\ref{gseries}) as a generalized sum.
	This leads to a possible way for the analytic continuation of zeta functions (see Theorem \ref{gsondow}).
\end{remark}

\begin{proof}[Proof of Theorem \ref{Euler1}.]
	We prove it from the induction on $k$. By manipulating the series (\ref{gseries}), we formally get
	\begin{equation}\label{add1}
		\begin{aligned}
			A_c
			=&b_1+\cdots+b_{c-1}-(c-1)b_c+b_{c+1}+\cdots+b_{2c-1}-(c-1)b_{2c}+\cdots\\
			=&\frac{1}{c}\left(cb_1+(c+1)b_2+\cdots+(2c-2)b_{c-1}-(c-1)^2b_c\right)
			\\&-\frac{1}{c}[(b_2+b_3+\cdots+b_{c}-(c-1)b_{c+1})+(b_3+b_4+\cdots+b_{c+1}-(c-1)b_{c+2})\\
			&\quad\quad+\cdots+(b_c+b_{c+1}+\cdots+b_{2c-2}-(c-1)b_{2c-1})\\
			&\quad\quad-(c-1)(b_{c+1}+b_{c+2}+\cdots+b_{2c-1}-(c-1)b_{2c})+\cdots]\\
			=&\frac{1}{c}\left(\sum_{i=0}^{c-2}(c+i)b_{i+1}-(c-1)^2b_c\right)-\frac{1}{c}\sum_{n=1}^{\infty}a_{c,n}\sum_{i=1}^{c}a_{c,i}b_{n+i}\\
			=&\frac{1}{c}\left(\sum_{i=0}^{c-2}(c+i)\Delta_c^0b_{i+1}-(c-1)^2\Delta_c^0b_c\right)-\frac{1}{c}\sum_{n=1}^{\infty}a_{c,n}\Delta_c^1b_n,
		\end{aligned}
	\end{equation}
	which is (\ref{anc}) in the case of $k=1$.	
	We assert that if the series $\sum_{n=1}^{\infty}a_{c,n}b_n$ is convergent, then the right hand side of  (\ref{add1}) is also convergent and (\ref{add1})  is established. 
	Indeed, let $$S_{c,n}=\frac{1}{c}\left(\sum_{i=0}^{c-2}(c+i)\Delta_c^0b_{i+1}-(c-1)^2\Delta_c^0b_c\right)-\frac{1}{c}\sum_{k=1}^{n}a_{c,k}\Delta_c^1b_k$$
	be the partial sum of the above series. By writing $n=cm+d$ with $m,d\in\mathbb{Z}$, $m\geqslant1$, $0\leqslant d< c$, we have
	\begin{equation}\label{add2}
		\begin{aligned}
			\sum_{k=1}^{n+1}a_{c,k}b_k-S_{c,n}
			=&\sum_{k=1}^{cm+1}a_{c,k}b_k-S_{c,cm}+\sum_{k=cm+2}^{cm+d+1}a_{c,k}b_k-\frac{1}{c}\sum_{k=cm+1}^{cm+d}a_{c,k}\Delta_c^1b_k\\
			=&\frac{1}{c}\left(-\sum_{i=2}^{c-1}(-c-i+1)b_{cm+i}+(c-1)^2b_{c(m+1)}\right)\\
			&+\sum_{k=cm+2}^{cm+d+1}a_{c,k}b_k-\frac{1}{c}\sum_{k=cm+1}^{cm+d}a_{c,k}\sum_{i=1}^{c}a_{c,i}b_{k+i},
		\end{aligned}
	\end{equation}
	which is a $\mathbb{Z}$-linear combination of  finite many terms $b_k$ $(cm+2\leqslant k\leqslant cm+c+d)$. Thus if the series $\sum_{n=1}^{\infty}a_{c,n}b_n$ converges, then $b_k\rightarrow 0$ as $k\rightarrow\infty$ and by (\ref{add2})
	$$\lim_{n\to\infty}\left(\sum_{k=1}^{n+1}a_{c,k}b_k-S_{c,n}\right)=0,$$
	which is equivalent to
	\begin{equation}
		\begin{aligned}
			A_c&=\sum_{n=1}^{\infty}a_{c,n}b_n\\
			&=\frac{1}{c}\left(\sum_{i=0}^{c-2}(c+i)\Delta_c^0b_{i+1}-(c-1)^2\Delta_c^0b_c\right)-\frac{1}{c}\sum_{n=1}^{\infty}a_{c,n}\Delta_c^1b_n.
		\end{aligned}
	\end{equation}
	
	If we assume the theorem is true for $k$, then for $k+1$ we have
	\begin{align*}
		A_c
		=&\sum_{j=0}^{k-1}\frac{(-1)^j}{c^{j+1}}\left(\sum_{i=0}^{c-2}(c+i)\Delta_c^jb_{i+1}-(c-1)^2\Delta_c^jb_c\right)+\sum_{n=1}^{\infty}\frac{(-1)^k}{c^k}a_{c,n}\Delta_c^kb_n\\
		=&\sum_{j=0}^{k-1}\frac{(-1)^j}{c^{j+1}}\left(\sum_{i=0}^{c-2}(c+i)\Delta_c^jb_{i+1}-(c-1)^2\Delta_c^jb_c\right)\\
		&+\frac{(-1)^k}{c^k}\left(\frac{1}{c}\left(\sum_{i=0}^{c-2}(c+i)\Delta_c^kb_{i+1}-(c-1)^2\Delta_c^kb_{c}\right)\right.\\&\quad\quad\quad\quad\quad\left.-\frac{1}{c}\sum_{n=1}^{\infty}a_{c,n}\sum_{i=1}^{c}a_{i,c}\Delta_c^kb_{n+i}\right)\\
		=&\sum_{j=0}^{k-1}\frac{(-1)^j}{c^{j+1}}\left(\sum_{i=0}^{c-2}(c+i)\Delta_c^jb_{i+1}-(c-1)^2\Delta_c^jb_c\right)\\
		&+\frac{(-1)^k}{c^{k+1}}\left(\sum_{i=0}^{c-2}(c+i)\Delta_c^kb_{i+1}-(c-1)^2\Delta_c^kb_{c}\right)+\frac{(-1)^{k+1}}{c^{k+1}}\sum_{n=1}^{\infty}a_{c,n}\Delta_c^{k+1}b_n\\
		=&\sum_{j=0}^{k}\frac{(-1)^j}{c^{j+1}}\left(\sum_{i=0}^{c-2}(c+i)\Delta_c^jb_{i+1}-(c-1)^2\Delta_c^jb_c\right)+\frac{(-1)^{k+1}}{c^{k+1}}\sum_{n=1}^{\infty}a_{c,n}\Delta_c^{k+1}b_n
	\end{align*}
	and the convergence is ensured from $k$ to $k+1$ by the same procedure as above.
	This completes our proof.
\end{proof}

From Theorem \ref{Euler1}, we get the following generalization of Sondow's result.
\begin{theorem}\label{gsondow}
	For $c\geq2$ and $k \geq 1$ the product
	\begin{equation}\label{zcs}
		\begin{aligned}
			\zeta_{(c)}(s)&=(1-c^{1-s})\zeta(s)\\
			&=\sum_{j=0}^{k-1}\frac{(-1)^j}{c^{j+1}}\left(\sum_{i=0}^{c-2}(c+i)\Delta_c^j(i+1)^{-s}-(c-1)^2\Delta_c^jc^{-s}\right)\\ &\quad+\lim_{N\rightarrow\infty}\sum_{n=1}^{cN}\frac{(-1)^k}{c^k}a_{c,n}\Delta_c^kn^{-s}
		\end{aligned}
	\end{equation}
	provides the analytic continuation of $\zeta(s)$ on the punctured half plane Re$(s)>1-k$, $s \neq 1$ where the infinite series converges absolutely and uniformly on compact sets to a holomorphic function. Moreover, except that the convergence will not be absolute in the strip $-k<\text{Re}(s) \leq 1-k$, this remains true for $k \geq 0$ and $\text{Re}(s)>-k$, $s\neq1$. Especially, taking $k=1$, we have
	\begin{equation}\label{k1}
		\begin{aligned}
			\zeta_{(c)}(s)&=(1-c^{1-s})\zeta(s)\\&=\frac{1}{c}\left(\sum_{i=0}^{c-2}(c+i)(i+1)^{-s}-(c-1)^2c^{-s}
			-\lim_{N\rightarrow\infty}\sum_{n=1}^{cN}a_{c,n}\sum_{i=1}^{c}a_{i,c}(n+i)^{-s}\right)
		\end{aligned}
	\end{equation}
	for Re$(s)>-1$, $s \neq 1$.
\end{theorem}
\begin{proof}
By setting  $b_n=n^{-s}$ in (\ref{anc}), for $\sigma=\text{Re}(s)>1$, we get
\begin{equation}\label{zc}
	\begin{aligned}
			\zeta_{(c)}(s)&=(1-c^{1-s})\zeta(s)\\&=\sum_{j=0}^{k-1}\frac{(-1)^j}{c^{j+1}}\left(\sum_{i=0}^{c-2}(c+i)\Delta_c^j(i+1)^{-s}-(c-1)^2\Delta_c^jc^{-s}\right)\\&\quad+\sum_{n=1}^{\infty}\frac{(-1)^k}{c^k}a_{c,n}\Delta_c^kn^{-s},
	\end{aligned}
\end{equation}
where $\Delta_c^0n^{-s}=n^{-s}$ and by (\ref{gc})  \begin{equation}\label{gc2}\Delta_c^kn^{-s}=\sum_{i=1}^{c}a_{c,i}\Delta_c^{k-1}(n+i)^{-s}\end{equation} for $k\geq 1$.

In the following, we investigate the convergent area in $\mathbb{C}$ for the series on the right hand side of (\ref{zc}).
 Let $(s)_0=1$ and $$(s)_k=s(s+1)\cdots(s+k-1)$$ for $k\geq 1$.
It can be checked directly  for $k\geq 1$ that
\begin{equation}\label{deltac} \Delta_{c}^{k}n^{-s}=(s)_kJ_k,\end{equation} where
	\begin{equation}\label{6+}J_k=\sum_{i_1,\ldots,i_k=1}^{c-1}\int_{i_k}^{c}\cdots\int_{i_1}^c(n+x_1+\cdots+x_k)^{-s-k}dx_1\cdots dx_k,
	\end{equation}
so we have the estimation 
	\begin{equation}\label{6}
		\abs{\Delta_c^kn^{-s}}\leq\frac{\abs{(s)_k}\left(\frac{c(c-1)}{2}\right)^k}{n^{\sigma+k}},
	\end{equation}
where $\sigma+k\geq0$ and $k=0,1,2,\ldots.$
% For a fixed $s$ in the half line $(-k, \infty)$ of the real axis,  by (\ref{deltac}), (\ref{6+}) and (\ref{6}),  $\{\Delta_c^kn^{-s}\}_{n=1}^{\infty}$ is a monotone sequence whose limit is 0,
% and by (\ref{C}), for any $N\in\mathbb{N}$, the partial sums $\sum_{n=1}^{N}a_{c,n}$ are bounded. So by Dirichlet's test,
% the series $$\sum_{n=1}^{\infty}a_{c,n}\Delta_c^kn^{-s}$$ converges and Theorem \ref{Euler1} is applied for $s\in (-k, \infty)$.
% Thus from the convergence theorem of Dirichlet series (see e.g. \cite[p. 233, Theorem 11.8, p. 235, Theorem 11.11 and p. 236, Theorem 11.12]{Apostol}), (\ref{zc}) provides an analytic continuation of $\zeta_{c}(s)$ on the half plane $\text{Re}(s)>-k$.
% 
% Then we consider the following  limit \begin{equation}\label{27} \sum_{n=1}^{\infty}a_{c,n}\Delta_c^kn^{-s}=\lim_{N\rightarrow\infty}\sum_{n=1}^{cN}a_{c,n}\Delta_c^kn^{-s}.\end{equation}
By (\ref{gc2}) we have
	\begin{equation}
	\begin{aligned}
		&\sum_{n=c+1}^{cN}a_{c,n}\Delta_{c}^{k}n^{-s}\\
		=&\sum_{m=2}^{N}\left(\Delta_{c}^{k}(cm-c+1)^{-s}+\cdots+\Delta_{c}^{k}(cm-1)^{-s}-(c-1)\Delta_{c}^{k}(cm)^{-s}\right)\\
		=&\sum_{m=2}^{N}\sum_{i=1}^{c}a_{c,i}\Delta_{c}^{k}(cm-c+i)^{-s}\\
				=&\sum_{m=2}^{N}\Delta_{c}^{k+1}(cm-c)^{-s}.
	\end{aligned}
	\end{equation}
	Thus from the estimation (\ref{6}), we get
	$$\left|\sum_{n=c+1}^{cN}a_{c,n}\Delta_{c}^{k}n^{-s}\right|=\left|\sum_{m=2}^{N}\Delta_{c}^{k+1}(cm-c)^{-s}\right|\leq\sum_{m=2}^{N}\frac{\abs{(s)_{k+1}}\left(\frac{c(c-1)}{2}\right)^{k+1}}{(cm-c)^{\sigma+k+1}},$$
	so in the half plane $\sigma>-k$, the series\begin{equation}\label{series}\lim_{N\rightarrow\infty}\sum_{n=1}^{cN}a_{c,n}\Delta_{c}^{k}n^{-s}=\sum_{n=1}^{c}a_{c,n}\Delta_{c}^{k}n^{-s}+\lim_{N\rightarrow\infty}\sum_{n=c+1}^{cN}a_{c,n}\Delta_{c}^{k}n^{-s}\end{equation} converges uniformly on any compact set. 
	
	In addition, by the estimation (\ref{6}) and $\abs{a_{c,n}}\leq c-1$ for $n\in\mathbb{N}$, we have 
	\begin{equation}\begin{aligned} \left|\sum_{n=1}^{cN}a_{c,n}\Delta_{c}^{k}n^{-s}\right|&\leq \sum_{n=1}^{cN}\left|a_{c,n}\Delta_{c}^{k}n^{-s}\right|\\
	&\leq (c-1)\sum_{n=1}^{cN}\frac{\abs{(s)_k}\left(\frac{c(c-1)}{2}\right)^k}{n^{\sigma+k}},\end{aligned}
	\end{equation} thus in the half plane $\sigma>1-k,$ the series (\ref{series}) converges absolutely and uniformly on any compact set, so does the right hand side of (\ref{zcs}).
        From these, we conclude that  (\ref{zcs}) provides the analytic continuation of $\zeta(s)$ as described by the theorem.
        \end{proof}
The above theorem implies the following new expressions for the special values $\zeta(-m)$, which generalize Sondow's formula (\ref{s1}).
\begin{corollary}\label{corollaryc} For $m=0,1,2,\ldots$, we have \begin{equation}\label{zcc}
			\begin{aligned} \zeta(-m)&=\left(1-c^{m+1}\right)^{-1}\zeta_{(c)}(-m)\\
			&=\frac{1}{1-c^{m+1}}\sum_{j=0}^{m}\frac{(-1)^j}{c^{j+1}}\left(\sum_{i=0}^{c-2}(c+i)\Delta_c^j(i+1)^{m}-(c-1)^2\Delta_c^jc^{m}\right).
			\end{aligned}
	\end{equation}
	\end{corollary}
	\begin{proof} In this case $(-m)_j=0$ for $j>m,$ so by (\ref{deltac}), $\Delta_c^jn^m=0$. Thus if letting  $s=-m$ in (\ref{zc}), then we get our result.
	\end{proof}

\section{Generalizations of Wallis' formula}

With the above preparations, we can extend  Yung, Sondow and Huylebrouck's methords to $\zeta_{(c)}(s)$
and obtain the following two generalizations of Wallis' formula.

\begin{theorem}\label{main}
For any integer $c\geq 2$ and Re$(s) > 0$, we have
\begin{equation}\label{Theorem2}
e^{\zeta_{(c)}^{\prime}(s)}=\prod_{n=1}^{\infty}n^{\frac{(-1)a_{c,n}}{n^s}},
\end{equation}
where the coefficients $a_{c,n}$ is defined as in (\ref{C}).
\end{theorem}

\begin{remark}
Setting $c=2$ in Theorem \ref{main}, we have $a_{2,n}=(-1)^{n-1},$
so (\ref{Theorem2}) becomes to 
\begin{equation}
e^{\eta^{\prime}(s)}=\frac{2^{\frac{1}{2^s}}\cdot 4^{\frac{1}{4^s}}\cdot 6^{\frac{1}{6^s}}\cdot\cdots}{3^{\frac{1}{3^s}}\cdot 5^{\frac{1}{5^s}}\cdot 7^{\frac{1}{7^s}}\cdot\cdots},
\end{equation}
which is equivalent to (\ref{Theorem1}) above.
\end{remark}

\begin{proof}[Proof of Theorem \ref{main}.]
Recall that 
$$\zeta_{(c)}(s)=(1-c^{1-s})\zeta(s).$$
Since the only pole for $\zeta(s)$ is at $s=1$
and the factor $1-c^{1-s}$ has a simple zero at $s=1$,
$\zeta_{(c)}(s)$ is analytic on the whole complex plane $\mathbb{C}$.
By (\ref{C}), for any $N\in\mathbb{N}$, the partial sums $\sum_{n=1}^{N}a_{c,n}$ are all bounded,
thus by Jensen-Cahen's theorem (see, e.g., Conrad's lecture note on Dirichlet series \cite[Theorem 9]{KC}), the series
$\sum_{n=1}^{\infty}\frac{a_{c,n}}{n^s}$ is convergent and analytic on the half plane Re$(s)>0$,
with its derivative there computable termwise.
So $$\zeta_{(c)}(s)=\sum_{n=1}^{\infty}\frac{a_{c,n}}{n^s}$$
is an identity of analytic functions for Re$(s)>0$
and in this area we have
\begin{equation*}
\begin{aligned}
\zeta_{(c)}^{\prime}(s)&=\sum_{n=1}^{\infty}a_{c,n}(-1)n^{-s}\log n\\
&=\log \prod_{n=1}^{\infty} n^{\frac{(-1)a_{c,n}}{n^s}}.
\end{aligned}
\end{equation*}
This is equivalent to
\begin{equation}
e^{\zeta_{(c)}^{\prime}(s)}=e^{\log \prod_{n=1}^{\infty} n^{\frac{(-1)a_{c,n}}{n^s}}}=\prod_{n=1}^{\infty}n^{\frac{(-1)a_{c,n}}{n^s}},
\end{equation}
which is our result.

\end{proof}
\begin{theorem}\label{re-1}
	For any integer $c\geq 2$ and Re$(s)>-1$, we have
	\begin{equation}\label{Theorem 6}
		e^{\zeta_{(c)}^{\prime}(s)}=\left(\prod_{i=2}^{c-1}i^{-\frac{c+i-1}{ci^s}}\right)\cdot c^{\frac{(c-1)^2}{c^{s+1}}}\lim_{N\rightarrow\infty}\prod_{n=1}^{cN}\prod_{i=1}^{c}(n+i)^{\frac{a_{c,n}a_{i,c}}{c(n+i)^s}},
	\end{equation}
	where the coefficients $a_{c,n}$ is defined as in (\ref{C}). Note that the first product is defined to be $1$ for $c=2$.
\end{theorem}
\begin{proof}
	Taking derivatives on  the both sides of (\ref{k1}), we have 
	\begin{equation}\label{k1'}
		\begin{aligned}
			\zeta^{\prime}_{(c)}(s)=&-\frac{1}{c}\left(\sum_{i=2}^{c-1}(c+i-1)i^{-s}\log i-(c-1)^2c^{-s}\log c\right)\\&+\frac{1}{c}\lim_{N\rightarrow\infty}\sum_{n=1}^{cN}a_{c,n}\sum_{i=1}^{c}a_{i,c}(n+i)^{-s}\log(n+i)\\
			=&\log\left(\left(\prod_{i=2}^{c-1}i^{-\frac{c+i-1}{ci^s}}\right)\cdot c^{\frac{(c-1)^2}{c^{s+1}}}\lim_{N\rightarrow\infty}\prod_{n=1}^{cN}\prod_{i=1}^{c}(n+i)^{\frac{a_{c,n}a_{i,c}}{c(n+i)^s}}\right),
		\end{aligned}
	\end{equation}
	so 
	\begin{equation}
		e^{\zeta_{(c)}^{\prime}(s)}=\left(\prod_{i=2}^{c-1}i^{-\frac{c+i-1}{ci^s}}\right)\cdot c^{\frac{(c-1)^2}{c^{s+1}}}\lim_{N\rightarrow\infty}\prod_{n=1}^{cN}\prod_{i=1}^{c}(n+i)^{\frac{a_{c,n}a_{i,c}}{c(n+i)^s}},
	\end{equation}
	which is the desired result.
\end{proof}

\section{Examples}\label{Exp}
In this section, we show some examples for our extensions of Wallis' formula in the above section.
First, we go to the case when $c=3$.
In this case, by (\ref{B}) we have
\begin{equation*}
\begin{aligned}
\zeta_{(3)}(s) 
&=1+\left(\frac{1}{2^s}-\frac{2}{3^s}+\frac{1}{4^s}\right)+\left(\frac{1}{5^s}-\frac{2}{6^s}+\frac{1}{7^s}\right)+\cdots\end{aligned}
\end{equation*}
and
\begin{equation*}
\begin{aligned}
 \zeta_{(3)}^{\prime}(s)&= 
\left((-1)2^{-s}\log2+2\cdot3^{-s}\log3+(-1)4^{-s}\log4\right)\\
&\quad+\left((-1)5^{-s}\log5+2\cdot 6^{-s}\log6+(-1)7^{-s}\log7\right)+\cdots,\\
&=\log \frac{3^{\frac{1}{3^s}}\cdot3^{\frac{1}{3^s}}}{2^{\frac{1}{2^s}}\cdot4^{\frac{1}{4^s}}} \frac{6^{\frac{1}{6^s}}\cdot6^{\frac{1}{6^s}}}{5^{\frac{1}{5^s}}\cdot7^{\frac{1}{7^s}}}\frac{9^{\frac{1}{9^s}}\cdot9^{\frac{1}{9^s}}}{8^{\frac{1}{8^s}}\cdot10^{\frac{1}{10^s}}}\cdots.
\end{aligned}
\end{equation*}
  So (\ref{Theorem2}) becomes to 
\begin{equation}\label{case3}
e^{\zeta_{(3)}^{\prime}(s)}=\frac{3^{\frac{1}{3^s}}\cdot3^{\frac{1}{3^s}}}{2^{\frac{1}{2^s}}\cdot4^{\frac{1}{4^s}}} \frac{6^{\frac{1}{6^s}}\cdot6^{\frac{1}{6^s}}}{5^{\frac{1}{5^s}}\cdot7^{\frac{1}{7^s}}}\frac{9^{\frac{1}{9^s}}\cdot9^{\frac{1}{9^s}}}{8^{\frac{1}{8^s}}\cdot10^{\frac{1}{10^s}}}\cdots,
\end{equation}
for Re$(s)>0$.
Furthermore, by \cite[Eq. (3)]{Choudhury}  we have the following Laurent series expansion of $\zeta(s)$ around $s=1$
\begin{equation}\label{1}
\zeta(s)=\frac{1}{s-1}+\sum_{n=0}^{\infty}\frac{(-1)^{n}\gamma_{n}}{n!}(s-1)^{n},
\end{equation}
where $\gamma_{n}$ is the Stieltjes constant in 1885 (see Stieltjes' original article \cite{St} and  Ferguson \cite{Fe})
and $\gamma=\gamma_{0}$ is the Euler-Mascheroni constant.
Since 
\begin{equation} 
\begin{aligned}
3^{1-s}&=e^{(1-s)\log 3}\\
&=1+(1-s)\log 3+\frac{(1-s)^{2}(\log 3)^{2}}{2!}+\frac{(1-s)^{3}(\log 3)^{3}}{3!}+\cdots,
\end{aligned}
\end{equation}
we get \begin{equation}\label{2} 1-3^{1-s}=(s-1)\log3+(s-1)(1-s)\frac{(\log 3)^{2}}{2!}+\cdots.\end{equation}
Then combing (\ref{1}) and (\ref{2})  we have
\begin{equation}
\begin{aligned}
\zeta_{(3)}(s)&=(1-3^{1-s})\zeta(s)\\
&=\left(\frac{1}{s-1}+\gamma-\gamma_{1}(s-1)+\cdots\right)\\
&\quad\times\left((s-1)\log3+(s-1)(1-s)\frac{(\log 3)^{2}}{2!}+\cdots\right)\\
&=\log 3+\left(\gamma \log 3-\frac{(\log 3)^{2}}{2}\right)(s-1)+\cdots
\end{aligned}
\end{equation}
and \begin{equation}\label{3}\zeta_{(3)}^{\prime}(1)=\log3\left(\gamma-\frac{\log 3}{2}\right).\end{equation}
In fact, for any integer $c\geq 2$, from the same reasoning we can also get
 $$\zeta_{(c)}^{\prime}(1)=\log c\left(\gamma-\frac{\log c}{2}\right).$$
 So (\ref{case3})  implies another generalization of Wallis'  formula,
 \begin{equation}\label{gen2}
 3^{\gamma-\frac{\log 3}{2}}=\frac{3^{\frac{1}{3}}\cdot3^{\frac{1}{3}}}{2^{\frac{1}{2}}\cdot4^{\frac{1}{4}}} \frac{6^{\frac{1}{6}}\cdot6^{\frac{1}{6}}}{5^{\frac{1}{5}}\cdot7^{\frac{1}{7}}}\frac{9^{\frac{1}{9}}\cdot9^{\frac{1}{9}}}{8^{\frac{1}{8}}\cdot10^{\frac{1}{10}}}\cdots.
 \end{equation}
Since $\zeta(2)=\frac{\pi^2}{6}$ and $$\zeta^{\prime}(2)=\frac{\pi^2}{6}\left(\gamma+\log 2\pi-12\log A\right),$$
where $\gamma$ is the Euler-Mascheroni constant and $A$ is the Glaisher-Kinkelin constant (see \cite[Eq. (13)]{Weisstein}),
by (\ref{A}) we have
\begin{equation*}
\begin{aligned}
\zeta_{(3)}^{\prime}(2)&=\frac{1}{3}\zeta(2)\log 3+\frac{2}{3}\zeta^{\prime}(2)\\
&=\frac{1}{3}\left(\frac{\pi^2}{6}\right)\log 3+\frac{2}{3}\left(\frac{\pi^2}{6}\left(\gamma+\log 2\pi-12\log A\right)\right)\\
&=\frac{\pi^2}{18}\log\left( 3\cdot\left(\frac{2\pi e^{\gamma}}{A^{12}}\right)^{2}\right).
\end{aligned}
\end{equation*}
So (\ref{case3}) also implies the following generalization of Wallis'  formula,
\begin{equation}\label{gen3}
\left(3\left(\frac{2\pi e^{\gamma}}{A^{12}}\right)^{2}\right)^{\frac{\pi^2}{18}}=\frac{3^{\frac{1}{3^2}}\cdot3^{\frac{1}{3^2}}}{2^{\frac{1}{2^2}}\cdot4^{\frac{1}{4^2}}} \frac{6^{\frac{1}{6^2}}\cdot6^{\frac{1}{6^2}}}{5^{\frac{1}{5^2}}\cdot7^{\frac{1}{7^2}}}\frac{9^{\frac{1}{9^2}}\cdot9^{\frac{1}{9^2}}}{8^{\frac{1}{8^2}}\cdot10^{\frac{1}{10^2}}}\cdots.\end{equation}
Although Theorem \ref{main} can not be applied to the point $s=0$, we can use Theorem \ref{re-1}. By (\ref{k1}) we have an identity of analytic functions on the half plane Re$(s)>-1,$
\begin{equation*}
	\begin{aligned}
		\zeta_{(3)}(s) 
		=&\frac{1}{3}(3\cdot1^{-s}+4\cdot2^{-s}-4\cdot3^{-s})\\&-\frac{1}{3}\lim_{N\rightarrow\infty}\sum_{n=1}^{3N}a_{n,3}((n+1)^{-s}+(n+2)^{-s}-2(n+3)^{-s}),
		\end{aligned}
\end{equation*}
so
\begin{equation*}
	\begin{aligned}
		\zeta_{(3)}^{\prime}(s)=&-\frac{1}{3}(4\cdot2^{-s}\log2-4\cdot3^{-s}\log3)+\frac{1}{3}\lim_{N\rightarrow\infty}\sum_{n=1}^{3N}a_{n,3}((n+1)^{-s}\log(n+1)\\&+(n+2)^{-s}\log(n+2)-2(n+3)^{-s}\log(n+3))\\
		=&\log\frac{3^{\frac{4}{3\cdot3^s}}\cdot2^{\frac{1}{3\cdot2^s}}\cdot3^{\frac{1}{3\cdot3^s}}\cdot3^{\frac{1}{3\cdot3^s}}\cdot4^{\frac{1}{3\cdot4^s}}\cdot6^{\frac{2}{3\cdot6^s}}\cdot6^{\frac{2}{3\cdot6^s}}\cdots}{2^{\frac{4}{3\cdot2^s}}\cdot4^{\frac{1}{3\cdot4^s}}\cdot4^{\frac{1}{3\cdot4^s}}\cdot5^{\frac{1}{3\cdot5^s}}\cdot5^{\frac{1}{3\cdot5^s}}\cdot4^{\frac{2}{3\cdot4^s}}\cdot5^{\frac{2}{3\cdot5^s}}\cdots}.
	\end{aligned}
\end{equation*}
Thus (\ref{Theorem 6}) becomes to
\begin{equation}\label{case4}
	e^{\zeta_{(3)}^{\prime}(s)}=\frac{3^{\frac{4}{3\cdot3^s}}\cdot2^{\frac{1}{3\cdot2^s}}\cdot3^{\frac{1}{3\cdot3^s}}\cdot3^{\frac{1}{3\cdot3^s}}\cdot4^{\frac{1}{3\cdot4^s}}\cdot6^{\frac{2}{3\cdot6^s}}\cdot6^{\frac{2}{3\cdot6^s}}\cdots}{2^{\frac{4}{3\cdot2^s}}\cdot4^{\frac{1}{3\cdot4^s}}\cdot4^{\frac{1}{3\cdot4^s}}\cdot5^{\frac{1}{3\cdot5^s}}\cdot5^{\frac{1}{3\cdot5^s}}\cdot4^{\frac{2}{3\cdot4^s}}\cdot5^{\frac{2}{3\cdot5^s}}\cdots}.
\end{equation}
On the other hand, since $\zeta(0)=-\frac{1}{2}$ and $\zeta^{\prime}(0)=-\frac{1}{2}\log 2\pi$ (see \cite[p. 1049, 9.542]{GR}),
by (\ref{A}) we have 
\begin{equation*}
	\begin{aligned}
		\zeta_{(3)}^{\prime}(0)&=3\zeta(0)\log 3+(-2)\zeta^{\prime}(0)\\
		&=3\left(-\frac{1}{2}\right)\log 3+(-2)\left(-\frac{1}{2}\log 2\pi\right)\\
		&=\log\frac{2\pi}{3^{\frac{3}{2}}},
	\end{aligned}
\end{equation*}
so (\ref{case4})  implies another generalization of Wallis'  formula,
\begin{equation}\label{gen1}
	\frac{2\pi}{3^{\frac{3}{2}}}=\frac{3^{\frac{4}{3}}}{2^{\frac{4}{3}}}
	\frac{2^{\frac{1}{3}}\cdot3^{\frac{1}{3}}\cdot3^{\frac{1}{3}}\cdot4^{\frac{1}{3}}\cdot6^{\frac{2}{3}}\cdot6^{\frac{2}{3}}}{4^{\frac{1}{3}}\cdot4^{\frac{1}{3}}\cdot5^{\frac{1}{3}}\cdot5^{\frac{1}{3}}\cdot4^{\frac{2}{3}}\cdot5^{\frac{2}{3}}}\cdots.\end{equation}
Finally, we need to remark here that  similar forms with (\ref{gen1}), (\ref{gen2}) and (\ref{gen3}) are also established for $c=4, 5, 6,\ldots$.
For example, if $c=4$, then (\ref{gen1}) becomes to
\begin{equation}\label{gen4}
\frac{(2\pi)^{\frac{3}{2}}}{16}=
\frac{4^{\frac{9}{4}}}{2^{\frac{5}{4}}\cdot3^{\frac{3}{2}}}
\frac{2^{\frac{1}{4}}\cdot3^{\frac{1}{4}}\cdot4^{\frac{1}{4}}\cdot3^{\frac{1}{4}}\cdot4^{\frac{1}{4}}\cdot5^{\frac{1}{4}}\cdot4^{\frac{1}{4}}\cdot5^{\frac{1}{4}}\cdot6^{\frac{1}{4}}\cdot8^{\frac{3}{4}}\cdot8^{\frac{3}{4}}\cdot8^{\frac{3}{4}}}
{5^{\frac{1}{4}}\cdot5^{\frac{1}{4}}\cdot5^{\frac{1}{4}}\cdot6^{\frac{1}{4}}\cdot6^{\frac{1}{4}}\cdot6^{\frac{1}{4}}\cdot7^{\frac{1}{4}}\cdot7^{\frac{1}{4}}\cdot7^{\frac{1}{4}}\cdot5^{\frac{3}{4}}\cdot6^{\frac{3}{4}}\cdot7^{\frac{3}{4}}}
\cdots.\end{equation}

\bibliography{central}

\end{document}